\documentclass{amsproc}
\usepackage{tikz-cd} 
\usepackage{float}
\usepackage{amsmath}
\usepackage{amsthm}
\usepackage{amsfonts}
\usepackage{mathtools}

\usepackage{lipsum}

\let\OLDthebibliography\thebibliography
\renewcommand\thebibliography[1]{
  \OLDthebibliography{#1}
  \setlength{\parskip}{0pt}
  \setlength{\itemsep}{0pt plus 0.3ex}
}

\newtheorem{thm}{Theorem}[section]
\newtheorem{defn}[thm]{Definition}
\newtheorem{lem}[thm]{Lemma}
\newtheorem{prop}[thm]{Proposition}
\newtheorem{cor}[thm]{Corollary}
\newtheorem{exmp}[thm]{Example}
\newtheorem{rem}[thm]{Remark}

\definecolor{cadmiumgreen}{rgb}{0.0, 0.42, 0.24}
\definecolor{darkpastelgreen}{rgb}{0.01, 0.75, 0.24}

\numberwithin{equation}{section}
\counterwithin{figure}{section}
\counterwithin{table}{section}

\newcommand{\map}[3]{{#1}:{#2}\longrightarrow{#3}}
\renewcommand{\a}{\alpha}
\renewcommand{\b}{\beta}

\newcommand{\len}[1]{\ell(#1)}
\newcommand{\N}{\mathbb{N}}
\newcommand{\pt}{\partial}

\newcommand{\Il}{\mathsf{L}}
\newcommand{\Ir}{\mathsf{R}}
\newcommand{\pow}[2]{{#1}^{#2}}

\newcommand{\emptypartition}{\varepsilon}

\newcommand{\Pset}{\mathcal{P}}
\newcommand{\D}{\mathcal{D}}
\newcommand{\df}{\partial}

\newcommand{\Nil}{\mathcal{N}}
\newcommand{\Fset}{\mathcal{F}}
\newcommand{\AP}{\mathcal{A}}
\newcommand{\BP}{\mathcal{B}}
\newcommand{\two}[1]{\mu_2(#1)}
\newcommand{\Wset}{\mathcal{W}}
\newcommand{\w}{\omega}
\newcommand{\supp}{\mathcal{S}}

\newcommand{\comm}{\sim_{\mathcal{C}}}
\newcommand{\ncomm}{\not\sim_{\mathcal{C}}}
\newcommand{\Orb}{\mathcal{O}}

\DeclareMathOperator{\im}{Im}
\DeclareMathOperator{\IS}{IS}

\newcommand{\RRset}{\mathcal{Q}}
\newcommand{\emptyf}{\emptypartition}
\newcommand{\burge}[1]{\Omega(#1)}
\newcommand{\pti}[2]{{#1}^{(#2)}}
\newcommand{\Des}[1]{\mathrm{Des}(#1)}
\newcommand{\des}[1]{\mathrm{des}(#1)}
\newcommand{\maj}[1]{\mathrm{maj}(#1)}

\usetikzlibrary{tikzmark}
\newcounter{pairctr}

\usepackage{amssymb}
\usepackage{graphicx}

\title{Jordan types for pairs of commuting nilpotent matrices: \\ A survey}

\author[T. Ko\v{s}ir]{Toma\v{z} Ko\v{s}ir}
\address{University of Ljubljana, Faculty of Mathematics and Physics, Ljubljana, Slovenia, and Institute of Mathematics, Physics and Mechanics, Ljubljana, Slovenia
}
\email{tomaz.kosir@fmf.uni-lj.si}

\thanks{The author acknowledges financial support from the Slovenian Research and Innovation Agency (research core funding No. P1-0448 until December 31, 2025, and research core funding P1-0222 and research project No. J1-50002 after January 1, 2026). }

\date{\today}

\begin{document}

\begin{abstract}
  The aim of the paper is to survey results on Jordan types for pairs of nilpotent commuting matrices. We also review recent proof of the Box Conjecture on Jordan types that have equal dense orbit in the nilpotent commutator. 
\end{abstract}

\maketitle

\begin{center}
    \textit{Dedicated to Tony Iarrobino for his 80th birthday.}  
\end{center}

\section{Introduction}

We survey results on possible pairs of integer partitions that occur as Jordan types of two commuting nilpotent matrices. We call such pairs \emph{commuting Jordan types} and the corresponding partitions \emph{commuting partitions}. Precise description of all possible commuting Jordan types appears to be a difficult task. The fact that it is field dependent was proved for finite fields by Britnell and Wildon \cite[Prop. 4.12]{BW}. We assume that the underlying field is an infinite field. The survey is a complement to Khatami's survey \cite{Kha-S}. 

The nilpotent commutator of a nilpotent matrix of Jordan type $P$ is an irreducible variety and therefore one of the nilpotent orbits has a dense intersection with it. We denote the Jordan type of the dense orbit by $\D(P)$ following Panyushev \cite{Pan}, who introduced map $\D$. Recently, the Box Conjecture formulated by Iarrobino, Khatami, Van Steirtenghem, and Zhao \cite{IKvSZ} on Jordan types in the inverse image $\D^{-1}(P) $ was proved in \cite{IKM}. The central part of the survey is dedicated to an outline of the proof. But first, we state some positive and some negative results on commuting Jordan types that are scattered in the literature. They were proved using tools of linear algebra and combinatorics. Section \ref{Dominant} is dedicated to  the dense orbit map $\D$ and its properties. We include also an observation that seems to be new:  Two non-equal Rogers-Ramanujan partitions are not commuting. A partition is \emph{super-distinct} or \emph{Rogers-Ramanujan} if its parts differ by at least $2$. We conclude the survey with a short list of open questions. 

Early references on commuting matrices can be traced to the end of the 19th century and the beginning of 20th century, e.g., \cite{Frob-1,Frob-2,Plem,Schur,Sylv,Tab-1,Tab-2}. For a more detailed discussion of the early developments we refer to \cite[pp. 93-94]{MacDu}. The study of varieties of commuting matrices can be traced to Motzkin and Taussky \cite{MoTa} and Gerstenhaber \cite{Ger}. See Guralnick \cite{Gur} for a discussion of the irreducibility of varieties of commuting pairs of matrices. Some interesting results on algebras of commuting matrices are found in \cite{SuTy}.

I met the problem of commuting Jordan types while working on my PhD thesis \cite{Kos}. First results on the nilpotent commutator were proved by Basili \cite{Bas-1,Bas-2} and Baranovsky \cite{Bar}. Later, Oblak studied the problem in her PhD thesis \cite{Obl-T} and published her results in \cite{Obl-1,Obl-2,Obl-3}. Independently, results related to the problem of commuting Jordan types were published by Panyushev \cite{Pan}, Premet \cite{Pre}, and Basili, Iarrobino, and Khatami \cite{BI,BIK,IK,Kha-1,Kha-2}. Recently, the topic of commuting Jordan types in artinian algebras has attracted substantial interest (see \cite{AIM-survey} for a survey).

The focus of Khatami's survey \cite{Kha-S} is on the results for the dense orbit map $\D$ before the proof of the Box Conjecture. Here we survey also other results on commuting pairs scattered in the literature and provide a review of the proof of the Box Conjecture. 

The Box Conjecture was communicated to me by Tony Iarrobino most likely by 2012. It represents a natural formulation of the structures that one gained understanding of through work of Basili, Iarrobino, Khatami, Oblak, Panyushev, and others. During her work on the PhD thesis \cite{Obl-T}, Oblak noticed that the inverse image $\D^{-1}(Q)$ for $Q=(p,q)$ has $(p-q-1)q$ elements (see \cite[Prop. 4.4]{Obl-3} for publication of the result). Honors student Zhao, who worked with Iarrobino in 2011/12 on a student research project, proposed to arrange partitions in $\D^{-1}(p,q)$ in a rectangular array. The authors of \cite{IKvSZ} developed these ideas further and published them as the Table Theorem \cite[Thms. 3.12 \& 3.19]{IKvSZ} and as the Box Conjecture \cite[Conj. 4.11]{IKvSZ}. The first version of the paper \cite{IKvSZ} was put on arXiv repository in 2014, four years prior to the publication of the paper.

The Box Theorem stands out as the central part of our current understanding of commuting Jordan types for pairs of nilpotent matrices. The Burge correspondence gives a natural connection with the semigroup theory as described later in the paper. The examples found recently in the joint work with Khatami \cite{Kha-Kos} indicate that the Burge correspondence is likely not sufficient for further progress on solving the problems of commuting Jordan types. Definitely, it is a useful tool to obtain new results. For instance, it was one of the tools used to prove the main theorem of \cite[Thm. 1.3]{BoIK2025} and it was used to find examples of pairs of Jordan types in $\D^{-1}(p,q)$ that do not commute in \cite{Kha-Kos}.

\section{Preliminaries}

We denote by $\Pset$ the set of all partitions of natural numbers and by $\Pset(n)$ the set of all partitions of natural number $n$. A partition $P\in\Pset$ is written as $P=(p_1,p_2,\ldots,p_k)$, where $p_i$ are arranged in nonincreasing order and  $p_k> 0$. We adjoin an arbitrary finite number of $0$ to $P$ on the right if needed. We also use the 'power' notation $P=(p_1^{m_1},p_2^{m_2},\ldots,p_r^{m_r})$, where $p_i$ are strictly decreasing, $p_r>0$, and all multiplicities $m_i$ are  positive integers.

We assume that $F$ is an infinite field. We denote by $F[x]$ the algebra of all polynomials with coefficients in $F$ and by $M_n(F)$ the algebra of all $n\times n$ matrices over $F$.

Suppose $B\in M_n(F)$ is a nilpotent matrix and $P=P_B=(p_1,p_2,\ldots,p_k)\in\Pset$ is its Jordan type, i.e., $p_1\ge p_2\ge\ldots\ge p_k>0$ are sizes of its Jordan blocks. The nilpotent commutator $\Nil_B$ of $B$ is defined by 
$$\Nil_B=\{A\in M_n(F):\,BA=AB,A\ \text{nilpotent}\}.$$ 

We are interested in pairs of partitions $(P,Q)\in\Pset^2$ that are Jordan types of two commuting nilpotent matrices. We use notation $P\comm Q$ for such pairs of partitions, i.e., we write $P\comm Q$ if there exist two commuting nilpotent matrices $A$ and $B$ such that $P_A=P$ and $P_B=Q$. We say that $P$ and $Q$ are \emph{commuting partitions} if $P\comm Q$.  Relation $\comm$ is symmetric and reflexive, but it is not transitive.

In connection with commuting Jordan types two classes of partitions play important role. A partition $Q=(q_1,q_2,\ldots,q_k)$ is called \emph{super-distinct or Rogers-Ramanujan} (RR for short) if its parts differ by at least $2$: $q_i-q_{i+1}\ge 2$ for $i=1,2,\ldots,k-1$. We denote the set of all RR partitions by $\RRset$.

On the opposite side a partition $P=(p_1,p_2,\ldots,p_k)$ is called \emph{almost rectangular} (AR for short) if its parts differ by at most $1$: $p_1-p_{k}\le 1$. 

\vspace{0.2cm}

Below are Ferrers diagrams of an RR partition (left) and an AR partition (right):

\vspace{0.2cm}

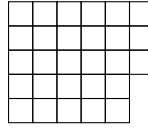
\begin{minipage}[t]{0.4\textwidth}
 Partition $(8,5,3,1)$\\
    \begin{tikzpicture}[scale=0.4] 
        \def\squaresize{0.8}
        \foreach \x in {0,...,7} {
            \pgfmathsetmacro{\currentx}{\x * \squaresize} 
            \draw (\currentx,0) rectangle ++(\squaresize,\squaresize); 
        }
        \foreach \x in {0,...,4} {
            \pgfmathsetmacro{\currentx}{\x * \squaresize} 
            \pgfmathsetmacro{\currenty}{-\squaresize} 
            \draw (\currentx,\currenty) rectangle ++(\squaresize,\squaresize); 
        }
        \foreach \x in {0,...,2} {
            \pgfmathsetmacro{\currentx}{\x * \squaresize} 
            \pgfmathsetmacro{\currenty}{-2 * \squaresize} 
            \draw (\currentx,\currenty) rectangle ++(\squaresize,\squaresize); 
        }
        \pgfmathsetmacro{\currenty}{-3 * \squaresize} 
        \draw (0,\currenty) rectangle ++(\squaresize,\squaresize); 
    \end{tikzpicture}
\end{minipage}
\begin{minipage}[t]{0.4\textwidth}
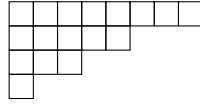
Partition $(6,6,6,5,5)=[28]^5$\\
    \begin{tikzpicture}[scale=0.4] 
        \def\squaresize{0.8}
        \foreach \x in {0,...,5} {
            \pgfmathsetmacro{\currentx}{\x * \squaresize} 
            \draw (\currentx,0) rectangle ++(\squaresize,\squaresize); 
        }
        \foreach \x in {0,...,5} {
            \pgfmathsetmacro{\currentx}{\x * \squaresize} 
            \pgfmathsetmacro{\currenty}{-\squaresize} 
            \draw (\currentx,\currenty) rectangle ++(\squaresize,\squaresize); 
        }
        \foreach \x in {0,...,5} {
            \pgfmathsetmacro{\currentx}{\x * \squaresize} 
            \pgfmathsetmacro{\currenty}{-2 * \squaresize} 
            \draw (\currentx,\currenty) rectangle ++(\squaresize,\squaresize); 
        }
        \foreach \x in {0,...,4} {
            \pgfmathsetmacro{\currentx}{\x * \squaresize} 
            \pgfmathsetmacro{\currenty}{-3 * \squaresize} 
            \draw (\currentx,\currenty) rectangle ++(\squaresize,\squaresize); 
        }
        \foreach \x in {0,...,4} {
            \pgfmathsetmacro{\currentx}{\x * \squaresize} 
            \pgfmathsetmacro{\currenty}{-4 * \squaresize} 
            \draw (\currentx,\currenty) rectangle ++(\squaresize,\squaresize); 
        }
    \end{tikzpicture}
\end{minipage}

\vspace{0.2cm}  

\section{Some miscellaneous results}\label{Misc}

In this section we present a collection of known results on Jordan types for commuting nilpotent matrices. Their proofs  involve  tools of linear algebra or combinatorics.

\subsection{One Jordan block}

We begin with the simplest case when $B$ is nonderogatory, i.e., $\dim\ker B=1$. Then $B$ is similar to a nilpotent Jordan block $J_n$ and its Jordan type is the one part partition $(n)$. The nilpotent commutator $\Nil_B$ is equal to the nil-algebra 
$$F_0[B]=\{p(B);\, p\in F[x],\, p(0)=0 \} $$ 
(see \cite[Thm. 9.1.4]{GLR}), i.e, $F_0(B)$ is the algebra of all polynomial functions in $B$ with the zero constant coefficient. Jordan types of $A\in\Nil_B$ are equal to Jordan types of powers $B^k$, $k=1,2,\ldots,n$. The corresponding partitions are exactly all almost rectangular partitions of $n$. Then, we write $P_{B^k} =[n]^k$. We summarize this case in the following statement.

\begin{prop}
    A partition $P\in\Pset(n)$ is commuting with partition $(n)$ if and only if $P=[n]^k $ for some $k$. So, partitions commuting with $(n)$ are exactly all almost rectangular partitions of $n$. There are $n$ of them.
\end{prop}

This result implies that in general, with increasing $n$, most partitions do not commute with the one part partition $(n)$, since the number of partitions grows as an exponential function of the square root of $n$ \cite[p. 69]{Andrews}. The smallest pair of non-commuting partitions occurs for $n=4$. Then partitions $(3,1)$ and $(4)$ are not commuting. (See Theorem \ref{RR do not commute}.)


\subsection{Universally commuting partitions}
\begin{defn}
A partition $P\in\Pset(n)$ is called \emph{universally commuting} if $P\comm Q$ for any partition $Q\in\Pset(n)$. 
\end{defn}

Oblak \cite[Thm. 2.4]{Obl-3} characterized all universally commuting partitions. The same result was proved independently by Britnell and Wildon \cite[Thm. 4.11]{BW}.

\begin{thm}\label{univ_commut}
For $n\leq 3$ all partitions are universally commuting, while for $n\geq 4$ only partitions of the form $(2^k,1^l)$ are universally commuting.
\end{thm}

\subsection{Conjugate partitions are commuting}
Given a partition $P=(p_1$, $p_2$, $\ldots,p_k)$, then the part $p_i'$ of its \emph{conjugate partition} $P^T=(p_1',p_2',\ldots,p_{p_1}')$ is equal to the number of parts of $P$ that are greater or equal to $i$. Baranovsky proved in the proof of \cite[Lem. 3]{Bar} the following fact. 

\begin{prop}
For each partition $P$ we have $P\comm P^T$, where $P^T$ is the conjugate partition of $P$. 
\end{prop}

Observe, that conjugate partitions of two commuting partitions might not be commuting. For example, partitions $(2,2,1,1)$ and $(2,1,1,1,1)$ are commuting by Theorem \ref{univ_commut}, while their conjugate partitions $(4,2)$ and $(5,1)$ are distinct RR partitions and therefore not commuting by Theorem \ref{RR do not commute}.

\subsection{Two parts partitions and commutativity}

Oblak \cite[Thm. 3.1]{Obl-3} proved that:

\begin{prop}\label{Two parts}
    Two distinct partitions of $n$ that have exactly two parts each are commuting if and only if $n$ is even, say $n=2k$ with $k\ge 2$, and one of partitions is equal to $(k,k)$ and the other is equal to $(k+1,k-1)$. 
\end{prop}

Oblak \cite[Prop. 3.10]{Obl-3} also proved:

\begin{prop}
If $Q=(q_1,q_2,\ldots)$ is a partition commuting with partition $P=(k,k)$ then either $Q=(2k)$ or $q_1\leq k+1$.   
\end{prop}

Recently, authors in \cite[Thm. 3]{BDKOS} characterized all partitions of $n=2k$ that commute with $(k,k)$.

\subsection{Refinements}

A \emph{subpartition} of a partition $P$ is a subsequence of $P$. The sum of a subpartition is the sum of its parts. If a subsequence $Q$ is an almost rectangular partition then we call it \emph{almost rectangular subpartition} (AR subpartition for short). 

\begin{defn}
If $P$ and $Q$ are two partitions of $n$ then $P$ is a \emph{refinement} of $Q$ if each part of $Q$ is a sum of a subpartition of $P$ and the subpartitions involved cover exactly once all of the parts of $P$. So, $P$ is a refinement of $Q$ if $Q=(q_1,q_2,\ldots,q_k)$ and $P$ is the disjoint union of partitions $P_1$, $P_2$,\ldots, $P_k$ such that the sum of all parts of $P_i$ is equal to $q_i$ for all $i$.

We say that $P$ is an \emph{almost rectangular refinement (an AR refinement)} of $Q$ if all subpartitions $P_i$ involved in the refinement are almost rectangular. 
\end{defn}

Note that an AR refinement itself is not an AR partition in general. It is a union of AR partitions. Any partition is trivially an (AR) refinement of itself.

For instance, $P=(5,4,4,2,2,1,1,1)$ is an AR refinement of $(13,7)$. One has $P=[13]^3\cup[7]^5$. It is also AR refinement of $(9,5,4,2)$ since $P$ is the disjoint union of AR subpartitions $(5,4)$, $(2,1,1,1)$, $(4)$ and $(2)$. On the other hand, since $P=(4,4,2,2)\cup(5,1,1,1)$ it follows that $P$ is a refinement of $(12,8)$. The fact that $P$ is not an AR refinement of $(12,8)$ follows since $P$ contains no AR subpartition of $12$.

\vskip 3mm

Britnell and Wildon \cite[Prop. 4.5]{BW} proved the following result.

\begin{thm}\label{thm:refinements}
If $P$ and $Q$ are two almost rectangular refinements of a partition $R\in\Pset$ then $P\comm Q$. 
\end{thm}

\begin{rem}
    Theorem \ref{thm:refinements} can be easily deduced also from the theory of the symmetric inverse semigroups $\IS(n)$ \cite{GaMa,Lip}. The semigroup $\IS(n)$ is represented by all $0$-$1$ square matrices of size $n$ that have at most one nonzero entry in each row and column \cite{Sol}. These matrices are also called rook matrices. For further results on  centralizers of  nilpotent elements in $\IS(n)$ see \cite{Kon-1,Kon-2,Lip}.
\end{rem}
\begin{rem}
    Notice that an RR partition is not an AR refinement of any other partition then itself. Furthermore, it is not enough to consider only pairs of AR refinements of RR partitions to obtain all pairs of AR refinements as in Theorem \ref{thm:refinements}. For instance, $P=(8,3^3)$ and $Q=(5,4,2^4)$ are both AR refinements of $R=(9,8)$ that is not an RR partition. Theorem \ref{thm:refinements} implies that partitions $(8,3^3)$ and $(5,4,2^4)$ are commuting.
\end{rem}

\subsection{Hook partitions and commutativity}
Authors of \cite{Kha-Kos} described all partitions that commute with a hook partition $(m,1^n)$ for $m\ge 3$ and $n\ge 1$. The special case $n=1$ was resolved earlier in \cite[Prop. 4.9]{BW}. The main result of \cite{Kha-Kos} states that any partition commuting with a hook partition $(n,1^m)$ contains as a subpartition an AR partition of either $n$, $n-1$ or $n-2$. In the case when the AR subpartition is an AR partition of $n$, the remaining subpartition is an arbitrary partition of $m$. In the other two cases, there are some technical conditions on the largest parts of the remaining subpartition of $m+1$ or $m+2$. 

\subsection{Union of commuting partitions is commuting}

Using direct sums of matrices we obtain also the following result.

\begin{prop}
If $(P_1,P_2)$ and $(Q_1,Q_2)$ are two pairs of commuting partitions then also $(P_1\cup Q_1,P_2\cup Q_2)$ is a pair of commuting partitions.    
\end{prop}

For instance, Proposition \ref{Two parts} implies that $(5,3)$ and $(4,4)$ are commuting partitions, and also that $(4,2)$ and $(3,3)$ are commuting. Then partitions $(5,4,3,2)$ and $(4,4,3,3)$ are commuting. More generally, one proves using Proposition \ref{Two parts} that $(k+3,k+2,k+1,k)$ and $(k+2,k+2,k+1,k+1)$ are commuting partitions for any $k\in\N$.

\section{The dominant partition}\label{Dominant}


The set $\Nil_B$ of all nilpotent matrices commuting with $B$ is an irreducible algebraic variety \cite[Lem. 2.3]{Bas-2}. Therefore, intersection of $\Nil_B$ with one of the nilpotent orbits under $Gl_n(F)$ conjugation is a dense open subset in the Zariski topology on $M_n(F)$. The Jordan type of the dense orbit is denoted by $\D(P)$, where $P$ is the Jordan type of $B$. The map $\D$ was introduced by Panyushev \cite{Pan}. Here the choice of letter can refer to the dense orbit. By Proposition \ref{D(P)-maximal},  $\D(P)$ is also the dominant partition (in the dominance order) among all partitions that are commuting with $P$. For each partition $P$ we obviously have $P\comm\D(P)$. 

  We write $P\succeq Q$ for the \emph{dominance order} (also called the Bruhat order) on partitions. For partitions $P=(p_1,p_2,\ldots)$ and $Q=(q_1,q_2,\ldots)$ we have $P\succeq Q$ if and only if  
$$\sum_{i=1}^k p_i\geq\sum_{i=1}^k q_i$$
for all $k$.

The image of the map $\D:\Pset\to\Pset$ is equal to the set $\RRset$ of all {RR partitions}, i.e., partitions whose parts differ by at least $2$ \cite[Thm. 1.12]{BI}. The map $\D$ is idempotent \cite[Thm. 6]{KO}. 

The following is a known fact. It is stated e.g. by Khatami \cite[Prop. 3.2]{Kha-S}.

\begin{prop}\label{D(P)-maximal}
    Suppose that $P$ and $Q$ are partitions of $n$. If $P\comm Q$ then  $\D(P)\succeq Q$.
\end{prop}

\begin{proof}
    The intersection of $\Nil_B$ with the nilpotent orbit $\Orb_{\D(P)}$ is a dense open subset of $\Nil_B$. Since $\Nil_B$ is irreducible we have 
    $$\Nil_B=\overline{\Orb_{\D(P)}\cap\Nil_B}\subset \overline{\Orb_{\D(P)}}\cap\Nil_B\subset\Nil_B.$$
    Here $\overline{X}$ is the Zariski closure of the set $X$. The above chain of containments is in fact a chain of equalities, and so  
    $$\Nil_B=\overline{\Orb_{\D(P)}\cap\Nil_B} = \overline{\Orb_{\D(P)}}\cap\Nil_B.$$
    This implies that if $\Orb_Q$ is any orbit such that $\Nil_B\cap\Orb_Q\neq\emptyset$, then $\Orb_Q\subset\overline{\Orb_{\D(P)}}$. By the Gerstenhaber-Hesselink Theorem \cite[Thm. 6.2.5]{CoMcG} we have $\D(P)\succeq Q$.
\end{proof}

\begin{cor}\label{Q-maximal}
    For each partition $Q\in\RRset$ we have $Q\succeq P$ for any $P\in\D^{-1}(Q)$.
\end{cor}

Among negative results, i.e., which pairs of partitions are not pairs of Jordan types of two commuting matrices, we state first an easy consequence of Proposition \ref{D(P)-maximal}.

\begin{thm}\label{RR do not commute}
    If $Q$ and $R$ are two distinct partitions in $\RRset$, then $Q\ncomm R$.
\end{thm}

\begin{proof}
    Since $\D$ is idempotent and $Q,R\in\im\D$, we have $Q=\D(Q)$ and $R=\D(R)$. 
    
    We assume that the nilpotent orbit $\Orb_Q$ intersects $\Nil_B$, where $R$ is the Jordan type of $B$. We want to obtain a contradiction. By Proposition \ref{D(P)-maximal} our assumption implies that $R=\D(R)\succeq Q$. By reversing the roles of $R$ and $Q$ we also obtain that $Q=\D(Q)\succeq R$, which is possible if and only if $Q=R$. This contradicts the assumption that $Q$ and $R$ are distinct.
\end{proof}

A partition $Q\in\RRset$ is the element with the least number of parts among all partitions in the inverse image $\D^{-1}(Q)$. For each $Q\in\RRset$ there is also a unique partition with the maximal number of parts in $\D^{-1}(Q)$. We denote this unique partition by $Q_{\textrm{mx}}$. Example 3 in \cite{IKM} shows that in general the partition $Q_{\textrm{mx}}$ is not the minimal partition in $\D^{-1}(Q)$ in the dominance order. (The same can be observed also in Table~5.1 below where partitions $(8,4,1^5)$ and $(8,3^2,1^3)$ are incomparable in the dominance order.) So, $Q_{\textrm{mx}}$ is not in general the minimal element of $\D^{-1}(Q)$ in the Bruhat order. An example in \cite[Example 6]{Kha-Kos} also shows that in general $Q_{mx} $ does not commute with all other partitions in $\D^{-1}(Q)$.


\section{The Box Theorem}

A conjecture regarding the structure of the inverse images $\mathcal{D}^{-1}(Q)$ for $Q \in \mathcal{Q}$ was proposed in \cite[Conj. 4.11]{IKvSZ}. The name the \emph{Box Conjecture} was given to the conjecture since it was conjectured that  the  elements of  $\D^{-1}(Q)$ for a partition $Q=(q_1,q_2,\ldots,q_k)\in\RRset$ can be arranged in a box (i.e. an array) of sizes $q_k\times (q_{k-1}-q_k-1)\times \cdots\times (q_1-q_2-1)$ such that the partition in the $(i_1,i_2,\cdots,i_k)$-th position has exactly $\sum_{j=1}^k i_j$ parts. The conjecture was proved in \cite{IKvSZ} for partitions with at most $2$ parts. 
In \cite{IKM}, the authors provide a proof of the conjecture for a general $Q\in\RRset$. The main tools in the proof are the Burge correspondence \cite{Bur-1,Bur-2} (cf. also \cite{AB,Bre}) between the set of all partitions $\Pset$ and binary words, and Shayman's description \cite{Sha-1,Sha-2} of all invariant subspaces of a nilpotent matrix. Here we provide an overview of the proof. For the precise statement of The Box Theorem proceed to Theorem \ref{BoxThm}.

\subsection{The Burge correspondence}

To introduce the correspondence in \cite{IKM} representation of partitions by their \emph{frequency vectors} is used. Here we briefly review this correspondence. It is very natural to use the frequency representation of partitions in connection with the Burge correspondence. However, when working with Jordan type, it is more common to represent a partition with its parts since they represent sizes of Jordan blocks. To accommodate both views we briefly review the Burge correspondence on frequency vectors and then also include a longer discussion using the parts representation of partitions. Thus, we complement the proof in \cite{IKM}.

For a partition $P$ we write $f_j=f_j(P)$, $j\in\N$, for the number of parts in $P$ that are equal to $j$. The \emph{frequency vector of} $P$ is then the sequence $f=f(P)=(f_1,f_2,\ldots)$. Here $f$ could be limited to the first $p_1$ elements, but for convenience we may augment it with an arbitrary finite sequence of zeros. We denote the set of all frequency vectors by $\Fset$. Equivalently, $\Fset$ is the set of all finite sequences of nonnegative integers. For an $f\in\Fset$ we denote by $P(f)$ the corresponding partition in $\Pset$, i.e., the partition whose frequency vector is equal to $f$. When working with $f \in \Fset$, we adopt the convention that $f_0=0$.

Next we introduce notation for some important invariants (statistics) of a partition $P\in\Pset$ using both types of representation of a partition -- by its parts and by its frequency vector. \emph{The size of a partition} $P=(p_1,p_2,\ldots,p_k)$ or \emph{of its frequency vector} $f=f(P)=(f_1,f_2,\ldots,f_{p_1})$ is equal to the sum of its parts
$$|P|=|f|=\sum_{j=1}^k p_j=\sum_{j=1}^{p_1} jf_j,$$ 
\emph{the length of a partition or of its frequency vector} is the number of its parts
$$\len{P}=\len{f}=k=\sum_{j=1}^{p_1}f_j. $$  
The \emph{empty partition} is the unique partition $\varepsilon\in\Pset$ of size (and length) $0$. The $2$-\emph{measure} $\two{P}=\two{f}$ of $P$ or of $f=f(P)$ is the maximum length of an RR subpartition of $P$, i.e., a subpartition with parts differing by at least $2$. Note that $\two{P}$ is also the minimal number of AR partitions needed to cover $P$, i.e., a minimal number of partitions in an AR refinement of $P$. 

The concept of $k$-measure for $k\ge 1$ was introduced in combinatorics by Andrews, Bhattacharjee and Dastidar \cite{ABD} to study sets of partitions. See \cite{ACL,Irv} for further results. The $k$-measure $\mu_k(P)$ of a partition $P$ is the length of the longest subsequence of $P$ whose parts differ by at least $k$.

Let us illustrate the notation with an example. 
\begin{exmp}
Suppose $P=(9,6,6,5,4,2,1,1,1,1)$. Then its frequency vector is $f(P)=(4,1,0,1,1,2,0,0,1)$, $|P|=36,\len{P}=10$ and $\mu_2(P)=4.$
\end{exmp}

The \emph{support} of $P\in\Pset$ or of the corresponding $f=f(P)\in\Fset$ is the set of all $j$ such that $j$ is a part in $P$, or equivalently, all $j$ such that $f_j\neq 0$. We denote it by $\supp(P)=\supp(f)$. 

A \emph{spread} of $f\in\Fset$ or of a partition $P=P(f)$ is defined as a maximal interval $[i,j]\subset\supp(f)=\supp(P)$. Finally, we define  sets 
\begin{align*}
\textstyle
	\Il(f) &= \Il(P) = \bigcup \{i, i+2, \ldots, i+2\lfloor\tfrac{j-i}{2}\rfloor\} \\
	\Ir(f) &= \Ir(P) = \bigcup \{j, j-2, \ldots, j-2\lfloor\tfrac{j-i}{2}\rfloor\}, 
\end{align*}
where the unions extend over all spreads $[i,j]$ of $f$ or $P=P(f)$. 
Observe that they are both of equal size which is equal to $\two{f}=\two{P}$. Note also that for a spread $f=[i,j]$ of odd length $j-i+1$ we have $\Il(f)=\Ir(f)$, while for a spread $f=[i,j]$ of even length we have $\Il(f)\cap\Ir(f)=\emptyset$ and $\Il(f)\cup\Ir(f)=[i,j]$.

Let us consider an example.

\begin{exmp}
The partition corresponding to frequency vector $$f=(4,2,1,0,0,3,2,1,2,0,0,2)$$ is  
$$P=P(f) = (12^2,9^2,8,7^2,6^3,3,2^2,1^4),$$ 
We have $|f|=|P|=93$ and $\len{f}=\len{P}=17$.  The spreads of $f$ are $\{1,2,3\}$, $\{6,7,8,9\}$ and $\{12\}$, so  $\Il(f)=\{1,3,6,8,12\}$, $\Ir(f)=\{1, 3, 7, 9, 12\}$ and $\two{f}=\two{P}=5$. 
\end{exmp}

We now partition $\Pset$ in two disjoint sets: $$\AP = \{P \in \Pset \,:\, 1 \not \in \Ir(P)\}\ \text{and}\ \BP = \{P \in \Pset \,:\, 1 \in \Ir(P)\}.$$ 
Note the equivalent characterization of the partition is given as follows:  $P \in \BP$ if and only if $1$ is contained in a spread of $f(P)$ of odd size.

We now  introduce two  transformations $\map{\a}{\Pset}{\AP}$ and $\map{\b}{\Pset}{\BP}$.

Each of the maps $\a$ and $\b$ acts as a sequence of raising operators on $P$ or $f=f(P)$.  The precise recipes are as follows:
\begin{itemize}
\item	$\a(P)$ is obtained from $P$ by replacing  one part of size $i$ with one part of size $i+1$ for each $i \in \Il(P)$. 

\item	$\b(P)$ is obtained from $P$ by adding one part of size $1$, and then acting as $\a$ on the partition $P'$ that is obtained from $P$ by omitting all the parts of size $1$.
In terms of the frequency vectors one has
$$\b(P) = (1^{f_1+1}, \a(P(f_2,f_3,\ldots))).$$
\end{itemize}

\begin{lem}\cite[Lem. 6]{IKM}\label{Lem:abpt}
 Maps $\a$ and $\b$ are bijections from $\Pset$ to $\AP$ and $\BP$, respectively. Moreover, $\AP\cup\BP=\Pset$ and $\AP\cap\BP=\emptyset$.
\end{lem}

Lemma \ref{Lem:abpt} implies that $\a$ and $\b$ together with $\a^{-1}$ and $\b^{-1}$ generate the so-called \emph{polycyclic (inverse) monoid on two generators} \cite[\S{9.3}]{Law}. We may join the inverses and define a new map $\partial:\Pset\to\Pset$:
$$\pt(P) = \begin{cases}
	\a^{-1}(P) & \text{if $P \in \AP$,} \\
	\b^{-1}(P) & \text{if $P \in \BP$.}
	\end{cases} $$ 
	  \mbox{}\vspace{-.5cm}

\begin{lem}\cite[Lem. 6]{IKM}\label{Lem:abpt-2}
 Let $\a,\b,\pt$ be defined on $\Pset$ as above.  Then map $\map{\pt}{\Pset}{\Pset}$ is 2-to-1.
\end{lem}

\begin{exmp}
    
Given a partition $P=(6,4,4,3,3,2,1)$ of $n=23$ we show how the maps $\a$ and $\b$ work:

\vspace{3mm}
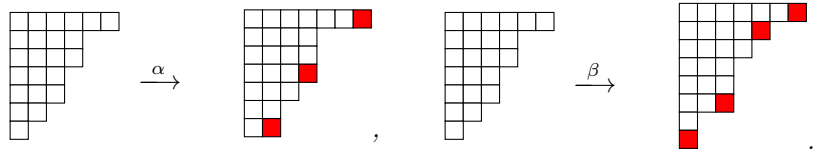
 \begin{minipage}[c]{0.125\textwidth}   \begin{tikzpicture}[scale=0.3] 
        \def\squaresize{0.8}

        \foreach \x in {0,...,5} {
            \pgfmathsetmacro{\currentx}{\x * \squaresize} 
            \ifnum\x=6 
                \filldraw[fill=white, draw=black] (\currentx,0) rectangle ++(\squaresize,\squaresize);
            \else
                \filldraw[fill=white, draw=black] (\currentx,0) rectangle ++(\squaresize,\squaresize);
            \fi
        }
        \foreach \x in {0,...,3} {
            \pgfmathsetmacro{\currentx}{\x * \squaresize} 
            \pgfmathsetmacro{\currenty}{-\squaresize} 
            \filldraw[fill=white, draw=black] (\currentx,\currenty) rectangle ++(\squaresize,\squaresize);
        }
        \foreach \x in {0,...,3} {
            \pgfmathsetmacro{\currentx}{\x * \squaresize} 
            \pgfmathsetmacro{\currenty}{-2 * \squaresize} 
            \filldraw[fill=white, draw=black] (\currentx,\currenty) rectangle ++(\squaresize,\squaresize);
        }
        \foreach \x in {0,...,2} {
            \pgfmathsetmacro{\currentx}{\x * \squaresize} 
            \pgfmathsetmacro{\currenty}{-3 * \squaresize} 
            \ifnum\x=3 
                \filldraw[fill=white, draw=black] (\currentx,\currenty) rectangle ++(\squaresize,\squaresize);
            \else
                \filldraw[fill=white, draw=black] (\currentx,\currenty) rectangle ++(\squaresize,\squaresize);
            \fi
        }
        \foreach \x in {0,...,2} {
            \pgfmathsetmacro{\currentx}{\x * \squaresize} 
            \pgfmathsetmacro{\currenty}{-4 * \squaresize} 
            \filldraw[fill=white, draw=black] (\currentx,\currenty) rectangle ++(\squaresize,\squaresize);
        }
        \foreach \x in {0,...,1} {
            \pgfmathsetmacro{\currentx}{\x * \squaresize} 
            \pgfmathsetmacro{\currenty}{-5 * \squaresize} 
            \ifnum\x=1 
                \filldraw[fill=white, draw=black] (\currentx,\currenty) rectangle ++(\squaresize,\squaresize);
            \else
                \filldraw[fill=white, draw=black] (\currentx,\currenty) rectangle ++(\squaresize,\squaresize);
            \fi
        }
        \pgfmathsetmacro{\currenty}{-6 * \squaresize} 
        \filldraw[fill=white, draw=black] (0,\currenty) rectangle ++(\squaresize,\squaresize);
    \end{tikzpicture}
    \end{minipage}
\begin{minipage}[c]{0.1\textwidth}
        $\xlongrightarrow{\a}$ 
        \end{minipage}
\begin{minipage}[c]{0.2\textwidth}
  \begin{tikzpicture}[scale=0.3] 
        \def\squaresize{0.8}

        \foreach \x in {0,...,6} {
            \pgfmathsetmacro{\currentx}{\x * \squaresize} 
            \ifnum\x=6 
                \filldraw[fill=red, draw=black] (\currentx,0) rectangle ++(\squaresize,\squaresize);
            \else
                \filldraw[fill=white, draw=black] (\currentx,0) rectangle ++(\squaresize,\squaresize);
            \fi
        }
        \foreach \x in {0,...,3} {
            \pgfmathsetmacro{\currentx}{\x * \squaresize} 
            \pgfmathsetmacro{\currenty}{-\squaresize} 
            \filldraw[fill=white, draw=black] (\currentx,\currenty) rectangle ++(\squaresize,\squaresize);
        }
        \foreach \x in {0,...,3} {
            \pgfmathsetmacro{\currentx}{\x * \squaresize} 
            \pgfmathsetmacro{\currenty}{-2 * \squaresize} 
            \filldraw[fill=white, draw=black] (\currentx,\currenty) rectangle ++(\squaresize,\squaresize);
        }
        \foreach \x in {0,...,3} {
            \pgfmathsetmacro{\currentx}{\x * \squaresize} 
            \pgfmathsetmacro{\currenty}{-3 * \squaresize} 
            \ifnum\x=3 
                \filldraw[fill=red, draw=black] (\currentx,\currenty) rectangle ++(\squaresize,\squaresize);
            \else
                \filldraw[fill=white, draw=black] (\currentx,\currenty) rectangle ++(\squaresize,\squaresize);
            \fi
        }
        \foreach \x in {0,...,2} {
            \pgfmathsetmacro{\currentx}{\x * \squaresize} 
            \pgfmathsetmacro{\currenty}{-4 * \squaresize} 
            \filldraw[fill=white, draw=black] (\currentx,\currenty) rectangle ++(\squaresize,\squaresize);
        }
        \foreach \x in {0,...,1} {
            \pgfmathsetmacro{\currentx}{\x * \squaresize} 
            \pgfmathsetmacro{\currenty}{-5 * \squaresize} 
            \ifnum\x=1 
                \filldraw[fill=white, draw=black] (\currentx,\currenty) rectangle ++(\squaresize,\squaresize);
            \else
                \filldraw[fill=white, draw=black] (\currentx,\currenty) rectangle ++(\squaresize,\squaresize);
            \fi
        }
           \foreach \x in {0,...,1} {
          \pgfmathsetmacro{\currentx}{\x * \squaresize} 
        \pgfmathsetmacro{\currenty}{-6 * \squaresize} 
         \ifnum\x=1 
                \filldraw[fill=red, draw=black] (\currentx,\currenty) rectangle ++(\squaresize,\squaresize);
            \else
                \filldraw[fill=white, draw=black] (\currentx,\currenty) rectangle ++(\squaresize,\squaresize);
            \fi    
            }
            \end{tikzpicture},
\end{minipage}
\begin{minipage}[c]{0.125\textwidth}
    \begin{tikzpicture}[scale=0.3] 
        \def\squaresize{0.8}

        \foreach \x in {0,...,5} {
            \pgfmathsetmacro{\currentx}{\x * \squaresize} 
            \ifnum\x=6 
                \filldraw[fill=white, draw=black] (\currentx,0) rectangle ++(\squaresize,\squaresize);
            \else
                \filldraw[fill=white, draw=black] (\currentx,0) rectangle ++(\squaresize,\squaresize);
            \fi
        }
        \foreach \x in {0,...,3} {
            \pgfmathsetmacro{\currentx}{\x * \squaresize} 
            \pgfmathsetmacro{\currenty}{-\squaresize} 
            \filldraw[fill=white, draw=black] (\currentx,\currenty) rectangle ++(\squaresize,\squaresize);
        }
        \foreach \x in {0,...,3} {
            \pgfmathsetmacro{\currentx}{\x * \squaresize} 
            \pgfmathsetmacro{\currenty}{-2 * \squaresize} 
            \filldraw[fill=white, draw=black] (\currentx,\currenty) rectangle ++(\squaresize,\squaresize);
        }
        \foreach \x in {0,...,2} {
            \pgfmathsetmacro{\currentx}{\x * \squaresize} 
            \pgfmathsetmacro{\currenty}{-3 * \squaresize} 
            \ifnum\x=3 
                \filldraw[fill=white, draw=black] (\currentx,\currenty) rectangle ++(\squaresize,\squaresize);
            \else
                \filldraw[fill=white, draw=black] (\currentx,\currenty) rectangle ++(\squaresize,\squaresize);
            \fi
        }
        \foreach \x in {0,...,2} {
            \pgfmathsetmacro{\currentx}{\x * \squaresize} 
            \pgfmathsetmacro{\currenty}{-4 * \squaresize} 
            \filldraw[fill=white, draw=black] (\currentx,\currenty) rectangle ++(\squaresize,\squaresize);
        }
        \foreach \x in {0,...,1} {
            \pgfmathsetmacro{\currentx}{\x * \squaresize} 
            \pgfmathsetmacro{\currenty}{-5 * \squaresize} 
            \ifnum\x=1 
                \filldraw[fill=white, draw=black] (\currentx,\currenty) rectangle ++(\squaresize,\squaresize);
            \else
                \filldraw[fill=white, draw=black] (\currentx,\currenty) rectangle ++(\squaresize,\squaresize);
            \fi
        }
        \pgfmathsetmacro{\currenty}{-6 * \squaresize} 
        \filldraw[fill=white, draw=black] (0,\currenty) rectangle ++(\squaresize,\squaresize);
    \end{tikzpicture}
\end{minipage}
\begin{minipage}[c]{0.1\textwidth}
        $\xlongrightarrow{\b}$ 
        \end{minipage}
\begin{minipage}[c]{0.2\textwidth}
\begin{tikzpicture}[scale=0.3] 
        \def\squaresize{0.8}

        \foreach \x in {0,...,6} {
            \pgfmathsetmacro{\currentx}{\x * \squaresize} 
            \ifnum\x=6 
                \filldraw[fill=red, draw=black] (\currentx,0) rectangle ++(\squaresize,\squaresize);
            \else
                \filldraw[fill=white, draw=black] (\currentx,0) rectangle ++(\squaresize,\squaresize);
            \fi
        }
        \foreach \x in {0,...,4} {
            \pgfmathsetmacro{\currentx}{\x * \squaresize} 
            \pgfmathsetmacro{\currenty}{-\squaresize} 
            \ifnum\x=4 
                \filldraw[fill=red, draw=black] (\currentx,\currenty) rectangle ++(\squaresize,\squaresize);
            \else
                \filldraw[fill=white, draw=black] (\currentx,\currenty) rectangle ++(\squaresize,\squaresize);
            \fi
        }        
                        \foreach \x in {0,...,3} {
            \pgfmathsetmacro{\currentx}{\x * \squaresize} 
        \pgfmathsetmacro{\currenty}{-2 * \squaresize} 
        \filldraw[fill=white, draw=black] (\currentx,\currenty) rectangle ++(\squaresize,\squaresize);
        }
                        \foreach \x in {0,...,2} {
                          \pgfmathsetmacro{\currentx}{\x * \squaresize} 
        \pgfmathsetmacro{\currenty}{-3 * \squaresize} 
        \filldraw[fill=white, draw=black] (\currentx,\currenty) rectangle ++(\squaresize,\squaresize);}
                \foreach \x in {0,...,2} {
                       \pgfmathsetmacro{\currentx}{\x * \squaresize} 
        \pgfmathsetmacro{\currenty}{-4 * \squaresize} 
        \filldraw[fill=white, draw=black] (\currentx,\currenty) rectangle ++(\squaresize,\squaresize);}

        \foreach \x in {0,...,2} {
            \pgfmathsetmacro{\currentx}{\x * \squaresize} 
            \pgfmathsetmacro{\currenty}{-5 * \squaresize} 
            \ifnum\x=2 
                \filldraw[fill=red, draw=black] (\currentx,\currenty) rectangle ++(\squaresize,\squaresize);
            \else
                \filldraw[fill=white, draw=black] (\currentx,\currenty) rectangle ++(\squaresize,\squaresize);
            \fi
        }

        \pgfmathsetmacro{\currenty}{-6 * \squaresize} 
        \filldraw[fill=white, draw=black] (0,\currenty) rectangle ++(\squaresize,\squaresize);
        \pgfmathsetmacro{\currenty}{-7 * \squaresize} 
        \filldraw[fill=red, draw=black] (0,\currenty) rectangle ++(\squaresize,\squaresize); 
    \end{tikzpicture}.
    \end{minipage}
\vspace{2mm}\\
 So $\a(P)=(7,4,4,4,3,2,2)$ and $\b(P)=(7,5,4,3,3,3,1,1)$. Observe that 
$|\a(P)|=26$ and $|\b(P)|=27$.
\end{exmp}

Note that we have $0 \leq |\pt(P)|<|P|$ for all $P \neq \emptyf$. Therefore, applying $\pt$ repeatedly to any $P \in \Pset$ will result in $\pt^{k}P=\emptyf$ for some least positive integer $k \geq 1$. This is used in the following definitions.

\begin{defn}
We define the \emph{Burge chain} of $P\in \Pset$ to be the sequence
$\pti{P}{0}, \pti{P}{1}, \pti{P}{2}, \ldots, \pti{P}{k}$, where we write $\pti{P}{i}=\partial^iP$ and $k$ is
 the smallest positive integer satisfying $\pti{P}{k}=\emptyf$.
The  \emph{Burge code} of $P$ is the binary word $\burge{P} = \w_1 \w_2 \cdots \w_{k+1} \in \{\a,\b\}^*$, the set of all words on two letters $\a$ and $\b$, defined by 
$$
	\w_i = \begin{cases}
	\a & \text{if $\pti{P}{i-1} \in \AP$,} \\
	\b & \text{if $\pti{P}{i-1} \in \BP$.}
	\end{cases}
$$ 
\end{defn}

So $P=\emptyf$ has trivial Burge chain $\emptyf$ and Burge code $\burge{\emptyf}=\a$, whereas the  chains/codes of all other partitions $P$ are of length at least 2. In fact since $\pt P = \emptyf$ only for $P = \emptyf$ and $P=(1)$, the chain of every $P \neq \emptyf$  ends with $\pti{P}{k-1} = (1) \in \BP$ and $\pti{P}{k} = \emptyf \in \AP$, and thus the code $\burge{P}$ ends with $\w_{k}\w_{k+1}=\b\a$.

Since $\pt|_{\AP}=\a^{-1}$ and $\pt|_{\BP}=\b^{-1}$, the definition of $\burge{P}=\w_1\cdots\w_n$   ensures that $\pti{P}{i-1}= (\w_i \circ \pt)(\pti{P}{i-1})=\w_i(\pti{P}{i})$ for all $i$. Thus we have
$$
P =\w_1(\pti{P}{1})=\w_1\w_2(\pti{P}{2})=\w_1\w_2\w_3(\pti{P}{3})=\cdots=(\w_1\w_2\cdots\w_n)(\emptyf),
$$
where the products of  $\w_i$ are to be interpreted as functional composition in the usual right-to-left order.

\begin{prop}\cite[Prop. 9]{IKM}
The Burge encoding $P \mapsto \burge{P}$ is a one-one correspondence between $\Pset$ and the set $\Wset=(\a^*\b)^*\a$ of all finite words on  $\{\a,\b\}$ that end with a singleton $\a$.
\end{prop}

We refer to \cite{Bre} for an extension of the one-to-one correspondence that is a three-way bijection between $\Pset$, $\Wset$ and a set of certain up-down paths.

\begin{exmp}
For partition $P=(7,6,4,4,2,2,1,1,1)$ of $n=28$ one has:
\vskip 6mm
   
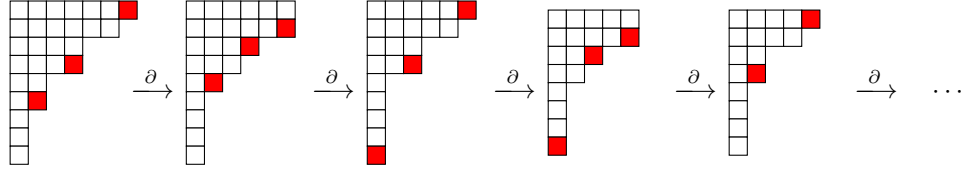
\begin{minipage}[c]{0.12\textwidth}
    \begin{tikzpicture}[scale=0.3] 
        \def\squaresize{0.8}

        \foreach \x in {0,...,6} {
            \pgfmathsetmacro{\currentx}{\x * \squaresize} 
            \ifnum\x=6 
                \filldraw[fill=red, draw=black] (\currentx,0) rectangle ++(\squaresize,\squaresize);
            \else
                \filldraw[fill=white, draw=black] (\currentx,0) rectangle ++(\squaresize,\squaresize);
            \fi
        }
        \foreach \x in {0,...,5} {
            \pgfmathsetmacro{\currentx}{\x * \squaresize} 
            \pgfmathsetmacro{\currenty}{-\squaresize} 
            \filldraw[fill=white, draw=black] (\currentx,\currenty) rectangle ++(\squaresize,\squaresize);
        }
        \foreach \x in {0,...,3} {
            \pgfmathsetmacro{\currentx}{\x * \squaresize} 
            \pgfmathsetmacro{\currenty}{-2 * \squaresize} 
            \filldraw[fill=white, draw=black] (\currentx,\currenty) rectangle ++(\squaresize,\squaresize);
        }
        \foreach \x in {0,...,3} {
            \pgfmathsetmacro{\currentx}{\x * \squaresize} 
            \pgfmathsetmacro{\currenty}{-3 * \squaresize} 
            \ifnum\x=3 
                \filldraw[fill=red, draw=black] (\currentx,\currenty) rectangle ++(\squaresize,\squaresize);
            \else
                \filldraw[fill=white, draw=black] (\currentx,\currenty) rectangle ++(\squaresize,\squaresize);
            \fi
        }
        \foreach \x in {0,...,1} {
            \pgfmathsetmacro{\currentx}{\x * \squaresize} 
            \pgfmathsetmacro{\currenty}{-4 * \squaresize} 
            \filldraw[fill=white, draw=black] (\currentx,\currenty) rectangle ++(\squaresize,\squaresize);
        }
        \foreach \x in {0,...,1} {
            \pgfmathsetmacro{\currentx}{\x * \squaresize} 
            \pgfmathsetmacro{\currenty}{-5 * \squaresize} 
            \ifnum\x=1 
                \filldraw[fill=red, draw=black] (\currentx,\currenty) rectangle ++(\squaresize,\squaresize);
            \else
                \filldraw[fill=white, draw=black] (\currentx,\currenty) rectangle ++(\squaresize,\squaresize);
            \fi
        }
        \pgfmathsetmacro{\currenty}{-6 * \squaresize} 
        \filldraw[fill=white, draw=black] (0,\currenty) rectangle ++(\squaresize,\squaresize);
        \pgfmathsetmacro{\currenty}{-7 * \squaresize} 
        \filldraw[fill=white, draw=black] (0,\currenty) rectangle ++(\squaresize,\squaresize);
        \pgfmathsetmacro{\currenty}{-8 * \squaresize} 
        \filldraw[fill=white, draw=black] (0,\currenty) rectangle ++(\squaresize,\squaresize);
    \end{tikzpicture}
\end{minipage}
\begin{minipage}[c]{0.05\textwidth}
    $\xlongrightarrow{\partial}$
    \end{minipage}
\begin{minipage}[c]{0.125\textwidth}
\begin{tikzpicture}[scale=0.3] 
        \def\squaresize{0.8}

        \foreach \x in {0,...,5} {
            \pgfmathsetmacro{\currentx}{\x * \squaresize} 
            \filldraw[fill=white, draw=black] (\currentx,0) rectangle ++(\squaresize,\squaresize);
        }
        \foreach \x in {0,...,5} {
            \pgfmathsetmacro{\currentx}{\x * \squaresize} 
            \pgfmathsetmacro{\currenty}{-\squaresize} 
            \ifnum\x=5 
                \filldraw[fill=red, draw=black] (\currentx,\currenty) rectangle ++(\squaresize,\squaresize);
            \else
                \filldraw[fill=white, draw=black] (\currentx,\currenty) rectangle ++(\squaresize,\squaresize);
            \fi
        }
        \foreach \x in {0,...,3} {
            \pgfmathsetmacro{\currentx}{\x * \squaresize} 
            \pgfmathsetmacro{\currenty}{-2 * \squaresize} 
            \ifnum\x=3 
                \filldraw[fill=red, draw=black] (\currentx,\currenty) rectangle ++(\squaresize,\squaresize);
            \else
                \filldraw[fill=white, draw=black] (\currentx,\currenty) rectangle ++(\squaresize,\squaresize);
            \fi
        }
        \foreach \x in {0,...,2} {
            \pgfmathsetmacro{\currentx}{\x * \squaresize} 
            \pgfmathsetmacro{\currenty}{-3 * \squaresize} 
            \filldraw[fill=white, draw=black] (\currentx,\currenty) rectangle ++(\squaresize,\squaresize);
        }
        \foreach \x in {0,...,1} {
            \pgfmathsetmacro{\currentx}{\x * \squaresize} 
            \pgfmathsetmacro{\currenty}{-4 * \squaresize} 
            \ifnum\x=1 
                \filldraw[fill=red, draw=black] (\currentx,\currenty) rectangle ++(\squaresize,\squaresize);
            \else
                \filldraw[fill=white, draw=black] (\currentx,\currenty) rectangle ++(\squaresize,\squaresize);
            \fi
        }
        \pgfmathsetmacro{\currenty}{-5 * \squaresize} 
        \filldraw[fill=white, draw=black] (0,\currenty) rectangle ++(\squaresize,\squaresize);
        \pgfmathsetmacro{\currenty}{-6 * \squaresize} 
        \filldraw[fill=white, draw=black] (0,\currenty) rectangle ++(\squaresize,\squaresize);
        \pgfmathsetmacro{\currenty}{-7 * \squaresize} 
        \filldraw[fill=white, draw=black] (0,\currenty) rectangle ++(\squaresize,\squaresize);
        \pgfmathsetmacro{\currenty}{-8 * \squaresize} 
        \filldraw[fill=white, draw=black] (0,\currenty) rectangle ++(\squaresize,\squaresize);
    \end{tikzpicture}
    \end{minipage}
\begin{minipage}[c]{0.05\textwidth}
    $\xlongrightarrow{\partial}$
    \end{minipage}
\begin{minipage}[c]{0.125\textwidth}
    \begin{tikzpicture}[scale=0.3] 
        \def\squaresize{0.8}

        \foreach \x in {0,...,5} {
            \pgfmathsetmacro{\currentx}{\x * \squaresize} 
            \ifnum\x=5 
                \filldraw[fill=red, draw=black] (\currentx,0) rectangle ++(\squaresize,\squaresize);
            \else
                \filldraw[fill=white, draw=black] (\currentx,0) rectangle ++(\squaresize,\squaresize);
            \fi
        }
        \foreach \x in {0,...,4} {
            \pgfmathsetmacro{\currentx}{\x * \squaresize} 
            \pgfmathsetmacro{\currenty}{-\squaresize} 
            \filldraw[fill=white, draw=black] (\currentx,\currenty) rectangle ++(\squaresize,\squaresize);
        }
        \foreach \x in {0,...,2} {
            \pgfmathsetmacro{\currentx}{\x * \squaresize} 
            \pgfmathsetmacro{\currenty}{-2 * \squaresize} 
            \filldraw[fill=white, draw=black] (\currentx,\currenty) rectangle ++(\squaresize,\squaresize);
        }
        \foreach \x in {0,...,2} {
            \pgfmathsetmacro{\currentx}{\x * \squaresize} 
            \pgfmathsetmacro{\currenty}{-3 * \squaresize} 
            \ifnum\x=2 
                \filldraw[fill=red, draw=black] (\currentx,\currenty) rectangle ++(\squaresize,\squaresize);
            \else
                \filldraw[fill=white, draw=black] (\currentx,\currenty) rectangle ++(\squaresize,\squaresize);
            \fi
        }
        \pgfmathsetmacro{\currenty}{-4 * \squaresize} 
        \filldraw[fill=white, draw=black] (0,\currenty) rectangle ++(\squaresize,\squaresize);
        \pgfmathsetmacro{\currenty}{-5 * \squaresize} 
        \filldraw[fill=white, draw=black] (0,\currenty) rectangle ++(\squaresize,\squaresize);
        \pgfmathsetmacro{\currenty}{-6 * \squaresize} 
        \filldraw[fill=white, draw=black] (0,\currenty) rectangle ++(\squaresize,\squaresize);
        \pgfmathsetmacro{\currenty}{-7 * \squaresize} 
        \filldraw[fill=white, draw=black] (0,\currenty) rectangle ++(\squaresize,\squaresize);
        \pgfmathsetmacro{\currenty}{-8 * \squaresize} 
        \filldraw[fill=red, draw=black] (0,\currenty) rectangle ++(\squaresize,\squaresize); 
    \end{tikzpicture}
    \end{minipage}
\begin{minipage}[c]{0.05\textwidth}
    $\xlongrightarrow{\partial}$
    \end{minipage}
\begin{minipage}[c]{0.125\textwidth}
\begin{tikzpicture}[scale=0.3] 
        \def\squaresize{0.8}

        \foreach \x in {0,...,4} {
            \pgfmathsetmacro{\currentx}{\x * \squaresize} 
            \filldraw[fill=white, draw=black] (\currentx,0) rectangle ++(\squaresize,\squaresize);
        }
        \foreach \x in {0,...,4} {
            \pgfmathsetmacro{\currentx}{\x * \squaresize} 
            \pgfmathsetmacro{\currenty}{-\squaresize} 
            \ifnum\x=4 
                \filldraw[fill=red, draw=black] (\currentx,\currenty) rectangle ++(\squaresize,\squaresize);
            \else
                \filldraw[fill=white, draw=black] (\currentx,\currenty) rectangle ++(\squaresize,\squaresize);
            \fi
        }
        \foreach \x in {0,...,2} {
            \pgfmathsetmacro{\currentx}{\x * \squaresize} 
            \pgfmathsetmacro{\currenty}{-2 * \squaresize} 
            \ifnum\x=2 
                \filldraw[fill=red, draw=black] (\currentx,\currenty) rectangle ++(\squaresize,\squaresize);
            \else
                \filldraw[fill=white, draw=black] (\currentx,\currenty) rectangle ++(\squaresize,\squaresize);
            \fi
        }
        \foreach \x in {0,...,1} {
            \pgfmathsetmacro{\currentx}{\x * \squaresize} 
            \pgfmathsetmacro{\currenty}{-3 * \squaresize} 
            \filldraw[fill=white, draw=black] (\currentx,\currenty) rectangle ++(\squaresize,\squaresize);
        }
        \pgfmathsetmacro{\currenty}{-4 * \squaresize} 
        \filldraw[fill=white, draw=black] (0,\currenty) rectangle ++(\squaresize,\squaresize);
        \pgfmathsetmacro{\currenty}{-5 * \squaresize} 
        \filldraw[fill=white, draw=black] (0,\currenty) rectangle ++(\squaresize,\squaresize);
        \pgfmathsetmacro{\currenty}{-6 * \squaresize} 
        \filldraw[fill=white, draw=black] (0,\currenty) rectangle ++(\squaresize,\squaresize);
        \pgfmathsetmacro{\currenty}{-7 * \squaresize} 
        \filldraw[fill=red, draw=black] (0,\currenty) rectangle ++(\squaresize,\squaresize); 
    \end{tikzpicture}
    \end{minipage}
\begin{minipage}[c]{0.05\textwidth}
    $\xlongrightarrow{\partial}$
    \end{minipage}
\begin{minipage}[c]{0.125\textwidth}
\begin{tikzpicture}[scale=0.3] 
        \def\squaresize{0.8}

        \foreach \x in {0,...,4} {
            \pgfmathsetmacro{\currentx}{\x * \squaresize}
            \ifnum\x=4 
                \filldraw[fill=red, draw=black] (\currentx,0) rectangle ++(\squaresize,\squaresize);
            \else
                \filldraw[fill=white, draw=black] (\currentx,0) rectangle ++(\squaresize,\squaresize);
            \fi
        }
        \foreach \x in {0,...,3} {
            \pgfmathsetmacro{\currentx}{\x * \squaresize}
            \pgfmathsetmacro{\currenty}{-\squaresize} 
            \filldraw[fill=white, draw=black] (\currentx,\currenty) rectangle ++(\squaresize,\squaresize);
        }
        \foreach \x in {0,...,1} {
            \pgfmathsetmacro{\currentx}{\x * \squaresize}
            \pgfmathsetmacro{\currenty}{-2 * \squaresize} 
            \filldraw[fill=white, draw=black] (\currentx,\currenty) rectangle ++(\squaresize,\squaresize);
        }
        \foreach \x in {0,...,1} {
            \pgfmathsetmacro{\currentx}{\x * \squaresize}
            \pgfmathsetmacro{\currenty}{-3 * \squaresize} 
            \ifnum\x=1 
                \filldraw[fill=red, draw=black] (\currentx,\currenty) rectangle ++(\squaresize,\squaresize);
            \else
                \filldraw[fill=white, draw=black] (\currentx,\currenty) rectangle ++(\squaresize,\squaresize);
            \fi
        }
        \pgfmathsetmacro{\currenty}{-4 * \squaresize}
        \filldraw[fill=white, draw=black] (0,\currenty) rectangle ++(\squaresize,\squaresize);
        \pgfmathsetmacro{\currenty}{-5 * \squaresize}
        \filldraw[fill=white, draw=black] (0,\currenty) rectangle ++(\squaresize,\squaresize);
        \pgfmathsetmacro{\currenty}{-6 * \squaresize}
        \filldraw[fill=white, draw=black] (0,\currenty) rectangle ++(\squaresize,\squaresize);
        \pgfmathsetmacro{\currenty}{-7 * \squaresize}
        \filldraw[fill=white, draw=black] (0,\currenty) rectangle ++(\squaresize,\squaresize);
    \end{tikzpicture}
    \end{minipage}
\begin{minipage}[c]{0.3\textwidth}
    $\xlongrightarrow{\partial}\quad\cdots$
    \end{minipage}

\vskip 7mm
 So the Burge code $\Omega(P)$ starts with $\a\a\b\b\a\ldots$ One can check that complete Burge code of $P$ 
 is $\Omega(P)=\a\a\b\b\a\a\b\b\b\a\a\b\b\b\b\a$.
 \label{exmp:burge}
\end{exmp}

\begin{exmp}

Word $\a\a\a\a\b\a$ corresponds to partition $(5)$: 
\vskip 7mm
\begin{minipage}[c]{0.02\textwidth}
 $\emptyset$
\end{minipage}
\begin{minipage}[c]{0.05\textwidth}
    $\xlongrightarrow{\beta}$
\end{minipage}
\begin{minipage}[c]{0.03\textwidth}
  \begin{tikzpicture}[scale=0.35] 
        \def\squaresize{0.8}

        \filldraw[fill=white, draw=black] ({0},{0}) rectangle ++(\squaresize,\squaresize);
    \end{tikzpicture}
\end{minipage}
\begin{minipage}[c]{0.05\textwidth}
    $\xlongrightarrow{\alpha}$
\end{minipage}
\begin{minipage}[c]{0.06\textwidth}
       \begin{tikzpicture}[scale=0.35] 
        \def\squaresize{0.8}

        \foreach \x in {0,...,1} {
            \pgfmathsetmacro{\currentx}{\x * \squaresize}
            \filldraw[fill=white, draw=black] ({\currentx},{0}) rectangle ++(\squaresize,\squaresize);
        }
    \end{tikzpicture}
\end{minipage}
\begin{minipage}[c]{0.05\textwidth}
    $\xlongrightarrow{\a}$
\end{minipage}
\begin{minipage}[c]{0.07\textwidth}
\begin{tikzpicture}[scale=0.35] 
        \def\squaresize{0.8}

        \foreach \x in {0,...,2} {
            \pgfmathsetmacro{\currentx}{\x * \squaresize}
            \filldraw[fill=white, draw=black] ({\currentx},{0}) rectangle ++(\squaresize,\squaresize);
        }
  
    \end{tikzpicture}
\end{minipage}
\begin{minipage}[c]{0.05\textwidth}
    $\xlongrightarrow{\a}$
\end{minipage}
\begin{minipage}[c]{0.09\textwidth}
   \begin{tikzpicture}[scale=0.35] 
        \def\squaresize{0.8}

        \foreach \x in {0,...,3} {
            \pgfmathsetmacro{\currentx}{\x * \squaresize}
            \filldraw[fill=white, draw=black] ({\currentx},{0}) rectangle ++(\squaresize,\squaresize);
        }
   
    \end{tikzpicture}
\end{minipage}
\begin{minipage}[c]{0.05\textwidth}
    $\xlongrightarrow{\alpha}$
\end{minipage}
\begin{minipage}[c]{0.11\textwidth}
   \begin{tikzpicture}[scale=0.35] 
        \def\squaresize{0.8}

        \foreach \x in {0,...,4} {
            \pgfmathsetmacro{\currentx}{\x * \squaresize}
            \filldraw[fill=white, draw=black] ({\currentx},{0}) rectangle ++(\squaresize,\squaresize);
        }
   
    \end{tikzpicture}
\end{minipage}

\vskip 5mm
Word $\b\b\b\b\a$ corresponds to partition $(1,1,1,1)$: 
\vskip 7mm

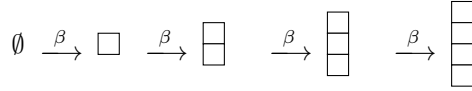
\begin{minipage}[c]{0.02\textwidth}
 $\emptyset$
\end{minipage}
\begin{minipage}[c]{0.05\textwidth}
    $\xlongrightarrow{\beta}$
\end{minipage}
\begin{minipage}[c]{0.04\textwidth}
  \begin{tikzpicture}[scale=0.35] 
        \def\squaresize{0.8}

        \filldraw[fill=white, draw=black] ({0},{0}) rectangle ++(\squaresize,\squaresize);
    \end{tikzpicture}
\end{minipage}
\begin{minipage}[c]{0.05\textwidth}
    $\xlongrightarrow{\b}$
\end{minipage}
\begin{minipage}[c]{0.06\textwidth}
       \begin{tikzpicture}[scale=0.35] 
        \def\squaresize{0.8}

        \filldraw[fill=white, draw=black] ({0},{0}) rectangle ++(\squaresize,\squaresize);
        \pgfmathsetmacro{\currenty}{-1 * \squaresize}
        \filldraw[fill=white, draw=black] ({0},{\currenty}) rectangle ++(\squaresize,\squaresize);
 
    \end{tikzpicture}
\end{minipage}
\begin{minipage}[c]{0.05\textwidth}
    $\xlongrightarrow{\b}$
\end{minipage}
\begin{minipage}[c]{0.06\textwidth}
 \begin{tikzpicture}[scale=0.35] 
        \def\squaresize{0.8}

        \filldraw[fill=white, draw=black] ({0},{0}) rectangle ++(\squaresize,\squaresize);
        \pgfmathsetmacro{\currenty}{-1 * \squaresize}
        \filldraw[fill=white, draw=black] ({0},{\currenty}) rectangle ++(\squaresize,\squaresize);
        \pgfmathsetmacro{\currenty}{-2 * \squaresize}
        \filldraw[fill=white, draw=black] ({0},{\currenty}) rectangle ++(\squaresize,\squaresize);
  
    \end{tikzpicture}
\end{minipage}
\begin{minipage}[c]{0.05\textwidth}
    $\xlongrightarrow{\b}$
\end{minipage}
\begin{minipage}[c]{0.12\textwidth}
  \begin{tikzpicture}[scale=0.35] 
        \def\squaresize{0.8}

        \filldraw[fill=white, draw=black] ({0},{0}) rectangle ++(\squaresize,\squaresize);
        \pgfmathsetmacro{\currenty}{-1 * \squaresize}
        \filldraw[fill=white, draw=black] ({0},{\currenty}) rectangle ++(\squaresize,\squaresize);
        \pgfmathsetmacro{\currenty}{-2 * \squaresize}
        \filldraw[fill=white, draw=black] ({0},{\currenty}) rectangle ++(\squaresize,\squaresize);
        \pgfmathsetmacro{\currenty}{-3 * \squaresize}
        \filldraw[fill=white, draw=black] ({0},{\currenty}) rectangle ++(\squaresize,\squaresize);
    \end{tikzpicture}
\end{minipage}

\vskip 5mm
Given word $\w=\a\b\b\a\b\a \in \Wset$ we draw Ferrers diagrams for emerging partitions and obtain $P(\w)=(5,2,1)$:

\vskip 2mm
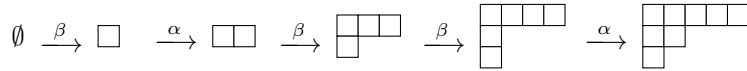
\begin{minipage}[c]{0.02\textwidth}
 $\emptyset$
\end{minipage}
\begin{minipage}[c]{0.05\textwidth}
    $\xlongrightarrow{\beta}$
\end{minipage}
\begin{minipage}[c]{0.05\textwidth}
  \begin{tikzpicture}[scale=0.35] 
        \def\squaresize{0.8}

        \filldraw[fill=white, draw=black] ({0},{0}) rectangle ++(\squaresize,\squaresize);
    \end{tikzpicture}
\end{minipage}
\begin{minipage}[c]{0.05\textwidth}
    $\xlongrightarrow{\alpha}$
\end{minipage}
\begin{minipage}[c]{0.06\textwidth}
       \begin{tikzpicture}[scale=0.35] 
        \def\squaresize{0.8}

        \foreach \x in {0,...,1} {
            \pgfmathsetmacro{\currentx}{\x * \squaresize}
            \filldraw[fill=white, draw=black] ({\currentx},{0}) rectangle ++(\squaresize,\squaresize);
        }
    \end{tikzpicture}
\end{minipage}
\begin{minipage}[c]{0.05\textwidth}
    $\xlongrightarrow{\beta}$
\end{minipage}
\begin{minipage}[c]{0.08\textwidth}
       \begin{tikzpicture}[scale=0.35] 
        \def\squaresize{0.8}

        \foreach \x in {0,...,2} {
            \pgfmathsetmacro{\currentx}{\x * \squaresize}
            \filldraw[fill=white, draw=black] ({\currentx},{0}) rectangle ++(\squaresize,\squaresize);
        }
        \pgfmathsetmacro{\currenty}{-\squaresize}
        \filldraw[fill=white, draw=black] ({0},{\currenty}) rectangle ++(\squaresize,\squaresize);
    \end{tikzpicture}
\end{minipage}
\begin{minipage}[c]{0.05\textwidth}
    $\xlongrightarrow{\beta}$
\end{minipage}
\begin{minipage}[c]{0.1\textwidth}
   \begin{tikzpicture}[scale=0.35] 
        \def\squaresize{0.8}

        \foreach \x in {0,...,3} {
            \pgfmathsetmacro{\currentx}{\x * \squaresize}
            \filldraw[fill=white, draw=black] ({\currentx},{0}) rectangle ++(\squaresize,\squaresize);
        }
        \pgfmathsetmacro{\currenty}{-1 * \squaresize}
        \filldraw[fill=white, draw=black] ({0},{\currenty}) rectangle ++(\squaresize,\squaresize);
        \pgfmathsetmacro{\currenty}{-2 * \squaresize}
        \filldraw[fill=white, draw=black] ({0},{\currenty}) rectangle ++(\squaresize,\squaresize);
    \end{tikzpicture}
\end{minipage}
\begin{minipage}[c]{0.05\textwidth}
    $\xlongrightarrow{\alpha}$
\end{minipage}
\begin{minipage}[c]{0.12\textwidth}
   \begin{tikzpicture}[scale=0.35] 
        \def\squaresize{0.8}

        \foreach \x in {0,...,4} {
            \pgfmathsetmacro{\currentx}{\x * \squaresize}
            \filldraw[fill=white, draw=black] ({\currentx},{0}) rectangle ++(\squaresize,\squaresize);
        }
        \foreach \x in {0,...,1} {
            \pgfmathsetmacro{\currentx}{\x * \squaresize}
            \pgfmathsetmacro{\currenty}{-1 * \squaresize}
            \filldraw[fill=white, draw=black] ({\currentx},{\currenty}) rectangle ++(\squaresize,\squaresize);
        }
        \pgfmathsetmacro{\currenty}{-2 * \squaresize}
        \filldraw[fill=white, draw=black] ({0},{\currenty}) rectangle ++(\squaresize,\squaresize);
    \end{tikzpicture}
\end{minipage}

\vskip 5mm
As the last example, we consider word $\w=\b\a\b\a\b\a \in \Wset$. Its corresponds to partition $P(\w)=(5,3,1)$:

\vskip 2mm
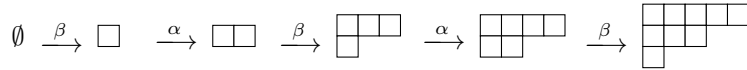
\begin{minipage}[c]{0.02\textwidth}
 $\emptyset$
\end{minipage}
\begin{minipage}[c]{0.05\textwidth}
    $\xlongrightarrow{\beta}$
\end{minipage}
\begin{minipage}[c]{0.05\textwidth}
  \begin{tikzpicture}[scale=0.35] 
        \def\squaresize{0.8}

        \filldraw[fill=white, draw=black] ({0},{0}) rectangle ++(\squaresize,\squaresize);
    \end{tikzpicture}
\end{minipage}
\begin{minipage}[c]{0.05\textwidth}
    $\xlongrightarrow{\alpha}$
\end{minipage}
\begin{minipage}[c]{0.06\textwidth}
       \begin{tikzpicture}[scale=0.35] 
        \def\squaresize{0.8}

        \foreach \x in {0,...,1} {
            \pgfmathsetmacro{\currentx}{\x * \squaresize}
            \filldraw[fill=white, draw=black] ({\currentx},{0}) rectangle ++(\squaresize,\squaresize);
        }
    \end{tikzpicture}
\end{minipage}
\begin{minipage}[c]{0.05\textwidth}
    $\xlongrightarrow{\beta}$
\end{minipage}
\begin{minipage}[c]{0.08\textwidth}
       \begin{tikzpicture}[scale=0.35] 
        \def\squaresize{0.8}

        \foreach \x in {0,...,2} {
            \pgfmathsetmacro{\currentx}{\x * \squaresize}
            \filldraw[fill=white, draw=black] ({\currentx},{0}) rectangle ++(\squaresize,\squaresize);
        }
        \pgfmathsetmacro{\currenty}{-\squaresize}
        \filldraw[fill=white, draw=black] ({0},{\currenty}) rectangle ++(\squaresize,\squaresize);
    \end{tikzpicture}
\end{minipage}
\begin{minipage}[c]{0.05\textwidth}
    $\xlongrightarrow{\alpha}$
\end{minipage}
\begin{minipage}[c]{0.1\textwidth}
   \begin{tikzpicture}[scale=0.35] 
        \def\squaresize{0.8}

        \foreach \x in {0,...,3} {
            \pgfmathsetmacro{\currentx}{\x * \squaresize}
            \filldraw[fill=white, draw=black] ({\currentx},{0}) rectangle ++(\squaresize,\squaresize);
        }
  \foreach \x in {0,...,1} {
            \pgfmathsetmacro{\currentx}{\x * \squaresize}
            \pgfmathsetmacro{\currenty}{-1 * \squaresize}
            \filldraw[fill=white, draw=black] ({\currentx},{\currenty}) rectangle ++(\squaresize,\squaresize);
        }
    \end{tikzpicture}
\end{minipage}
\begin{minipage}[c]{0.05\textwidth}
    $\xlongrightarrow{\beta}$
\end{minipage}
\begin{minipage}[c]{0.12\textwidth}
   \begin{tikzpicture}[scale=0.35] 
        \def\squaresize{0.8}

        \foreach \x in {0,...,4} {
            \pgfmathsetmacro{\currentx}{\x * \squaresize}
            \filldraw[fill=white, draw=black] ({\currentx},{0}) rectangle ++(\squaresize,\squaresize);
        }
        \foreach \x in {0,...,2} {
            \pgfmathsetmacro{\currentx}{\x * \squaresize}
            \pgfmathsetmacro{\currenty}{-1 * \squaresize}
            \filldraw[fill=white, draw=black] ({\currentx},{\currenty}) rectangle ++(\squaresize,\squaresize);
        }
        \pgfmathsetmacro{\currenty}{-2 * \squaresize}
        \filldraw[fill=white, draw=black] ({0},{\currenty}) rectangle ++(\squaresize,\squaresize);
    \end{tikzpicture}
\end{minipage}
\end{exmp}

We now describe how size, length and 2-measure are affected by $\a$, $\b$ and $\pt$.  

\begin{lem}\cite[Lem. 11]{IKM}
\label{lem:stats}
For any $P \in \Pset$ we have: 
\begin{enumerate}
\item	$\displaystyle
	\len{\a (P)} =
	\len{P},$ $\displaystyle
	\len{\b (P)} =
	\len{P} + 1 ,$ and \newline $\displaystyle
	\len{\pt P} =
	\begin{cases} 
	\len{P} - 1 &\text{if $P \in \BP$,} \\
	\len{P} &\text{if $\pt P \in \AP$,}\\
	\end{cases}$
\item	$\displaystyle |\a(P)| = |P| + \two{P},$ $\displaystyle |\b( P)| =
\begin{cases}
    |P| + \two{P} + 1 &\text{if $ P\in\AP $,} \\
	|P| + \two{P} &\text{if $P\in\BP$,}
	\end{cases},$  and \newline $\displaystyle |\pt P| = |P| - \two{P}$,

\item 	$\displaystyle
	\two{\a( P)} =
	\two{P} ,$ $\displaystyle
	\two{\b( P)} =
	\begin{cases}
	\two{P} + 1 &\text{if $P \in \AP$,} \\
	\two{P} &\text{if $ P \in \BP$,}
	\end{cases}$ and \newline $\displaystyle
	\two{\pt P} =
	\begin{cases}
	\two{P} - 1 &\text{if $P \in \BP$ and $\pt P \in \AP$,} \\
	\two{P} &\text{otherwise.}
	\end{cases}$
\end{enumerate}
\end{lem}

\subsection{The Descent Map}
\label{sec:universal}


\begin{defn}
 For $\w=\w_1\w_2\cdots\w_n \in \Wset$ its  {\emph{descent set}}  $\Des{\w}$ is the set of all indices $i$ for which $\w_i\w_{i+1}=\b\a$.
\end{defn}

 So, $\Des{\w}$ records the positions of descents in $\w$ relative to the ordering  $\b > \a$.  
These indices differ pairwise by at least $2$.

 \begin{exmp}\label{Exmp:Descent}
 Consider $$\w =\b{\color{red}\mathbf{|}}\a\b{\color{red}|}\a\a\a\b\b\b{\color{red}|}\a.$$
 Vertical lines ${\color{red}|}$ denote positions of descents, that is, places where $\b$ preceeds $\a$. 
The descent set is $$\Des{\w}=\{1,3,9\}.$$   
 \end{exmp}

\begin{defn}
The {\emph{descent number}} and {\emph{major index}} of $\w$ are 
$$\des{\w}=|\Des{\w}| \text{ and }\maj{\w}:=\sum_{i \in \Des{\w}} i.$$
\end{defn}

\begin{prop}\cite[Prop. 12]{IKM}
\label{prop:stats}
Let  $P \in \Pset$ and let $\w = \burge{P}$.  Then: 
\begin{enumerate}
\item	$\len{P} = $ \#\,occurrences of $\b$ in $\w$,
\item	$|P|=\maj{\w}$,
\item	$\two{P} = \des{\w} =$ \#\,occurrences of $\b\a$ in $\w$.
\end{enumerate}
\end{prop}

\begin{exmp}
For $\w =\b\a\b\a\a\a\b\b\b\a$ we have
 $P=P(\w)=(5,3,2,2,1)$.
Word 
 $\w$ contains $5$ copies of $\b$, hence  
 $\len{P}=5$. From its descent set in Example \ref{Exmp:Descent} we conclude
 $$|P|=13=\maj{\w}, \ \two{P}=3=\des{\w}.$$ 
\end{exmp}

\begin{exmp}
The partition $P=(7,6,4,4,2,2,1,1,1)$ in Example~\ref{exmp:burge} has Burge code $\w =\a\a\b\b\a\a\b\b\b\a\a\b\b\b\b\a $. Note that $\w$ contains $9=\len{P}$ copies of $\b$. Moreover, we have $\Des{\w}=\{4,9,15\}$,  so $\maj{\w}=28=|P|$ and  $\des{\w}=3=\two{P}$. 
\end{exmp}

\begin{defn}
For a partition $P$, let $P-1$ be the {\emph{reduced partition}} obtained by subtracting 1 from each part of $P$ and eliminating any resulting zeros. 
\end{defn}
\vskip 5mm
\begin{exmp}
For example: If $P=(6,4^3,2^4,1^3)$ then $P-1=(5,3^3,1^4)$.     
\end{exmp}

The following is the crucial characterization of the descent map that is used in the proof of the Box Theorem.
\vskip 5mm
\begin{thm}\cite[Lem. 15]{IKM}
The \emph{descent map}
 $P \mapsto \Des{P}$ is a size-preserving function from $\Pset$ to $\RRset$ satisfying $\Des{\pt P}=\Des{P}-1$.  
 It is the unique such function.
\end{thm}

 It is  possible to characterize partitions in $\RRset$ with a number of equivalent statements \cite[Prop. 23]{IKM}. One of them is by checking the Burge code $\Omega(P)$. We have $P\in\RRset$ if and only if $\burge{P}$ does not contain a substring $\b\b$. 
 
 In the statement of the Box Theorem we will use the common notation $\a^k$ instead of  string with $k$ symbols equal to $\a$. The same applies for $\b$, too.

\subsection{The Theorem}\label{Box_Theorem}

The proof of the Box Theorem is now completed by showing that the descent map and $\D$ coincide. The map $\D$ is obviously size preserving. The fact that $\D(\pt P)=\D(P)-1$ now hinges on the description of all invariant subspaces of a nilpotent matrix given by Shayman \cite{Sha-1,Sha-2}.

First, we use an easy consequence of Shayman's results.

\begin{prop}\cite[Prop. 19]{IKM}\label{W ia an image}
Each invariant subspace of $B$ is equal to the image of a matrix $A$ that commutes with $B$. 
\end{prop}

This allows us to describe the Jordan type of a restriction of $B$ to its generic invariant subspace.

\begin{thm}\cite[Thm. 23]{IKM}\label{partial(P) = P(B_W)}
    Suppose that $B$ is a nilpotent matrix with its Jordan type $P=P(B)$, that $A$ is a generic nilpotent matrix commuting with $B$ and that $W=\im A$ is its image. Then the Jordan type of the restriction $B|_W$ is given by $\df P$.
\end{thm}

This result implies the equality $\D(P)=\Des{P}$. The statement of the Box Theorem is then the following:

\begin{thm}\cite[Cor. 2]{IKM}\label{BoxThm}
Let $Q=(q_1,q_2,\ldots,q_k) \in \RRset$ and set $\delta_1 = q_k$ and $\delta_i = q_{k-i+1}-q_{k-i+2}-1$ for $2 \leq i \leq k$.  Then $\D^{-1}(Q)$ is of size $\delta_1\delta_2 \cdots\delta_k$ and consists of precisely those partitions whose Burge code is of the form
\begin{equation}
\label{eq:code}
\a^{\delta_1-i_1}\b^{i_1}\a^{\delta_2-i_2+1}\b^{i_2}
\a^{\delta_3-i_3+1} \b^{i_3} \cdots \a^{\delta_k-i_k+1} \b^{i_k}\a,
\end{equation}
where $(i_1,\ldots,i_k) \in [1,\delta_1] \times [1,\delta_2] \times \cdots \times [1,\delta_k]$.  Moreover, the partition determined by~\eqref{eq:code} has exactly $\sum_j i_j$ parts.
\end{thm}

Observe that for $Q \in \RRset$ we have $\D(Q)=Q$. So $Q\in\D^{-1}(Q) $ and it corresponds to taking $(i_1,\ldots,i_k)=(1,\ldots,1)$ in Theorem~\ref{BoxThm}.

\begin{exmp}\label{exmp-3}
For $Q=(9,6,2) \in \RRset$ we have $(\delta_1,\delta_2,\delta_3)=(2,3,2)$, so $\D^{-1}(Q)$ contains $12$ partitions indexed by coordinates in the box $[1,2] \times [1,3] \times [1,2]$. These are displayed in Table~5.1. \end{exmp}

\begin{table}[H]
\label{Figure1}
\centering
{
\begin{tabular}{l|l|l|l}
$(i_1,i_2,i_3)$ &  Burge code $\w$ &  partition $\Omega^{-1}(\w)$ & \# parts\\
\hline
  $ (1, 1, 1) $ & $ \a\b\a\a\a\b\a\a\b\a $ & 
  $ ( 9,6,2) $ & $3 $ \\  
  $ (2, 1, 1) $ & $ \b\b\a\a\a\b\a\a\b\a $ &
  $ ( 9,6,1^2) $ & $4 $ \\  
  $ (1, 2, 1) $ & $ \a\b\a\a\b\b\a\a\b\a $ & 
  $ ( 9,4,2^2) $ & $4 $ \\  
  $ (2, 2, 1) $ & $ \b\b\a\a\b\b\a\a\b\a $ & 
  $ ( 9,3^2,1^2) $ & $5 $ \\  
  $ (1, 3, 1) $ & $ \a\b\a\b\b\b\a\a\b\a $ & 
  $ ( 9,4,{2},1^2) $ & $5 $ \\  
  $ (2, 3, 1) $ & $ \b\b\a\b\b\b\a\a\b\a $ & 
  $ ( 9,4,\pow{1}{4}) $ & $6 $ \\  
  $ (1, 1, 2) $ & $ \a\b\a\a\a\b\a\b\b\a $ &
  $ ( 8,4,{3},{2}) $ & $4 $ \\  
  $ (2, 1, 2) $ & $ \b\b\a\a\a\b\a\b\b\a $ & 
  $ ( 8,{4},{3},1^2) $ & $5 $ \\  
  $ (1, 2, 2) $ & $ \a\b\a\a\b\b\a\b\b\a $ & 
  $ ( 8,4,\pow{2}{2},1) $ & $5 $ \\  
  $ (2, 2, 2) $ & $ \b\b\a\a\b\b\a\b\b\a $ & 
  $ ( 8,3^2,\pow{1}{3}) $ & $6 $ \\    
  $ (1, 3, 2) $ & $ \a\b\a\b\b\b\a\b\b\a $ & 
  $ ( 8,4,{2},\pow{1}{3}) $ & $6 $ \\  
  $ (2, 3, 2) $ & $ \b\b\a\b\b\b\a\b\b\a $ & 
  $ ( 8,4,\pow{1}{5}) $ & $7 $ \\   
 \end{tabular}
 }
 \label{fig:boxexample}
 \caption{The elements of $\D^{-1}(9,6,2)$ and the corresponding Burge codes, indexed by their coordinates $(i_1,i_2,i_3) \in [1,2] \times [1,3] \times [1,2]$.}
 \end{table}

\begin{rem}
    Observe that each syllable $\a^{\delta_i-i_i+1}\b^{i_i} $ (or $\a^{\delta_i-i_1}\b^{i_1}$) in the Burge code increases the dimension of the Box by $1$. Here, for $\delta_i=1$ the dimension is 'virtual', since we just have one 'slice' of the Box in the $i$-th direction. The extremal situation occurs for the Burge code $(\b\a)^k$ for some $k\in\N$. This Burge code corresponds to the partition $Q^{(k)} =(2k-1,2k-3,\ldots,3,1)$ that has all $\delta_i=1$. The partitions $Q^{(k)}$, $k\in\N$, are the only partitions $Q$ such that $\D^{-1}(Q)$ is the singleton set $\{Q\}$. Observe that $\left|Q^{(k)}\right|=k^2$ and $\len{Q^{(k)}}=\two{Q^{(k)}}=k$.
\end{rem}

Basili \cite[Thm. 1.2]{Bas-I} proved that the  process described by Oblak \cite{Obl-1,Obl-3} for each partition $P$ gives $\D(P)$. The iterative process reduces $P$ by the longest chain of $\D(P)$ in each step. Meanwhile, the Burge process described above reduces $P$ to $\pt P$  by reducing each almost rectangular part of $P$ by $1$ block in a backward manner. For precise relation between the two processes confer \cite[\S{4}]{IKM}. Another proof of the Basili Theorem is given in \cite[Thm. 34]{IKM}.

\section{Some Open Questions}

Despite recent progress made in understanding commuting Jordan types, several fundamental questions remain open:

\begin{itemize}
    \item \textbf{Complete Characterization of Commuting Jordan Types:} The central open problem is to {determine all pairs of commuting Jordan types} $(P, Q)$. While specific cases and properties have been identified, a general, explicit characterization or an algorithm to test commutativity for arbitrary pairs of partitions is still elusive.

    \item \textbf{Asymptotic proportion of commuting pairs:} No results are known on the asymptotic behaviour of the proportion of commuting partitions with increasing $n$. For $n$ small, one can count all pairs of commuting partitions using the results of section \ref{Misc} and \ref{Dominant}. If we denote by $C(n)$ the number of all pairs of commuting partitions of $n$ and by $A(n)$ the number of all pairs of partitions of $n$, what can one say about the behaviour of the quotient $\frac{C(n)}{A(n)}$ as $n$ grows?

    \item \textbf{Algebraic structure behind the Burge correspondence and more generally, 
    behind the commuting Jordan types:} The Burge correspondence has proven effective in describing the structure of $\mathcal{D}^{-1}(Q)$ as a ``box''. In \cite{Kha-Kos} the authors give examples of two partitions in the same box $\D^{-1}(Q)$ that do not commute. For instance, partitions $P_1=(7,1^2)$ and $P_2=(3^2,1^3)$ do not commute but we have $\D(P_1)=\D(P_2)=(7,2)$. Moreover, another example in \cite{Kha-Kos} shows that even partition $P_{mx}$ does not have to commute with all other partitions $P$ with $\D(P)=\D(P_{mx})$.  So, partitions with equal descent set need not commute. Can one decide from the Burge correspondence which pairs of partitions in a ''box'' are not commuting? 
    Are there any results in (inverse) semigroup theory that explain this phenomenon? And more generally, are there other algebraic constructions that govern the commutativity of Jordan types? 

\item \textbf{Stratification of orbits for pairs of commuting nilpotent matrices:} The previous problem has its geometric counterpart. Now, the more specific open problem is a generalization of the Gerstenhaber-Hesselink Theorem \cite[Thm. 6.2.5]{CoMcG} to the nilpotent commutator of a nilpotent matrix. Specifically, the question is which orbits are in the closure of the intersection of a given nilpotent orbit with the nilpotent commutator. In this direction, a conjecture on the stratification for a nilpotent commutator of a nilpotent matrix whose Jordan type is an RR partition was formulated by Boij, Iarrobino, and Khatami \cite[Conj. 4.2]{BoIK2025}. They provide the stratification for the case of an RR partition with two parts \cite[Thm. 1.3 and Cor. 3.8]{BoIK2025}. A more general question is to provide the Jordan types stratification of the variety of pairs of nilpotent commuting matrices.

    \item \textbf{Commuting Jordan types for other simple Lie algebras:} The study of commuting nilpotent elements and their Jordan types is primarily developed for the general linear Lie algebra $\mathfrak{gl}_n$ (type A). An intriguing question is whether there exists an analogous ``Burge correspondence'' or a similar combinatorial framework for describing commuting nilpotent elements in other simple Lie algebras (e.g., types B, C, D, etc.). This would extend the rich interplay between combinatorics and representation theory to a broader class of algebraic structures. Note that Premet \cite{Pre} showed that the variety of pairs of nilpotent elements in other simple Lie algebras is no longer irreducible. Some results on universally commuting Jordan types for other simple Lie algebras are given in Goddard \cite[Ch. 4]{God}.
\end{itemize}

These open questions and recent results, e.g. in \cite{BDKOS,Kha-Kos}, highlight the complexity and depth of the problem of commuting Jordan types, suggesting that further research will likely involve a blend of algebraic geometry, representation theory, and combinatorial methods.

\vskip 2mm

\textbf{Acknowledgment.} The author is grateful to Leila Khatami and Ganna Kudryavtseva for useful comments. He also thanks the anonymous referees for their careful reading and constructive comments, which strengthened the presentation and clarified several points. 

\vskip 2mm

\textbf{Declaration of generative AI and AI-assisted technologies.}
No AI tools were used during the preparation of this manuscript.
\bibliographystyle{plain}

\end{document}